\documentclass[a4paper]{scrartcl}

\usepackage{amsthm}
\usepackage{amsmath}
\usepackage{amssymb}
\usepackage[all,cmtip]{xy}
\usepackage{color}
\usepackage[T1]{fontenc}
\usepackage{mathrsfs}
\usepackage{enumerate}

\theoremstyle{plain}
\newtheorem{thm}{Theorem}
\newtheorem{lem}[thm]{Lemma}
\newtheorem{prop}[thm]{Proposition}

\theoremstyle{definition}
\newtheorem{defn}[thm]{Definition}
\newtheorem{exmp}[thm]{Example}

\newtheorem{rem}[thm]{Remark}

\newtheorem*{conventionsrings}{Conventions on rings}
\newtheorem*{conventionscategories}{Conventions on categories}

\DeclareMathOperator{\ob}{ob}
\DeclareMathOperator{\mor}{mor}

\DeclareMathOperator{\Aut}{Aut}

\DeclareMathOperator{\id}{id}

\DeclareMathOperator{\Supp}{Supp}
\DeclareMathOperator{\Ann}{Ann}

\DeclareMathOperator{\Top}{Top}
\DeclareMathOperator{\Ring}{Ring}
\DeclareMathOperator{\op}{op}

\DeclareMathOperator{\opsqr}{sqr}
\DeclareMathOperator{\opsqrt}{sqrt}
\DeclareMathOperator{\opabs}{abs}

\begin{document}

\title{Simple Skew Category Algebras Associated with Minimal Partially Defined Dynamical Systems}

\date{}

\author{Patrik Lundstr\"{o}m\footnote{Address: University West, Department of Engineering Science, SE-46186 Trollh\"{a}ttan, Sweden, E-mail: Patrik.Lundstrom@hv.se} \and Johan \"{O}inert\footnote{Address: Department of Mathematical Sciences, University of Copenhagen, Universitetsparken 5, DK-2100 Copenhagen \O, Denmark, E-mail: oinert@math.ku.dk}}


\maketitle

\begin{abstract}
In this article, we continue our study
of category dynamical systems, that is
functors $s$ from a category $G$ to $\Top^{\op}$,
and their corresponding skew category algebras.
Suppose that the
spaces $s(e)$, for $e \in \ob(G)$,
are compact Hausdorff. 
We show that if
(i) the skew category algebra is simple, then
(ii) $G$ is inverse connected,
(iii) $s$ is minimal and
(iv) $s$ is faithful.
We also show that if 
$G$ is a locally abelian groupoid,
then (i) is equivalent to (ii), (iii) and (iv).
Thereby, we generalize results
by \"{O}inert for skew group algebras to a large class of 
skew category algebras.
\end{abstract}


\pagestyle{headings}

\section{Introduction}

Ever since the classical papers on ergodic theory
and $C^*$-crossed products
by Murray and von Neumann (see e.g. \cite{mur36},
\cite{mur43} and \cite{von61}),
we have known that there is a connection
between topological properties of spaces and algebraical
properties of rings.
It has been observed by several authors that
there is a link between 
freness of topological dynamical systems and
ideal intersection properties of $C^*$-algebras
(see e.g. Zeller-Meier \cite{zel68}, Effros, Hahn \cite{effros67},
Elliott \cite{elliot80}, Archbold,
Quigg, Spielberg \cite{archbold93, quigg92, spielberg91},
Kawamura, Kishimoto, Tomiyama \cite{kawamura90, kishimoto81, tomiyama92}).
A lot of attention has also been given the 
connection between minimality of topological
dynamical systems and simplicity of 
the corresponding $C^*$-algebras
(see e.g. \cite{davidson96}, \cite{power78},
\cite{tomiyama87}, \cite{tomiyama92} and 
\cite{williams07}).
To be more precise,
suppose that $X$ is a topological space and
$s : X \rightarrow X$ is a continuous function;
in that case the pair $(X,s)$ is called a
\emph{topological dynamical system}.
A subset $Y$ of $X$ is called \emph{invariant} if
$s(Y) \subseteq Y$. The topological dynamical system
$(X,s)$ is called \emph{minimal} if there is no invariant
closed nonempty proper subset of $X$.
An element $x \in X$ is called
\emph{periodic} if there is a positive integer $n$
such that $s^n(x)=x$; an element of $X$ which is not periodic
is called \emph{aperiodic}.
Recall that the topological dynamical system
$(X,s)$ is called {\it topologically free}
if the set of aperiodic elements of $X$ is dense in $X$.
Now suppose that $s$ is a homeomorphism
of a compact Hausdorff space $X$.
Denote by $C(X)$ the unital $C^*$-algebra of continuous complex-valued functions
on $X$ endowed with the supremum norm, pointwise addition and multiplication,
and pointwise conjugation as involution.
The map $\sigma_s : C(X) \rightarrow C(X)$
which to a function $f \in C(X)$ associates $f \circ s \in C(X)$
is then an automorphism of $C(X)$.
The action of $\sigma_s$ on $C(X)$ extends in a unique
way to a strongly continuous representation $\sigma : {\Bbb Z} \rightarrow \Aut(C(X))$
subject to the condition that $\sigma(1) = \sigma_s$,
namely $\sigma(k) = \sigma_s^k$, for $k \in {\Bbb Z}$.
In that case, the associated transformation group
$C^*$-algebra $C^*(X,s)$ can be constructed
(see e.g. \cite{bla06} or \cite{pedersen79} for the details).
In 1978 Power showed the following
elegant result connecting
simplicity of $C^*(X,s)$ to minimality of $s$.

\begin{thm}[Power \cite{power78}]\label{powertheorem}
If $s$ is a homemorphism of a compact
Hausdorff
space $X$ of infinite cardinality,
then $C^*(X,s)$ is simple
if and only if $(X,s)$ is minimal.
\end{thm}

Inspired by Theorem \ref{powertheorem},
the second author of the current article
has in \cite{oinert09} and \cite{oinertarxiv11} 
shown analogous results (see Theorem \ref{maintheorem})
relating properties of an arbitrary topological dynamical system $(X,s)$,
where $s$ is a group homomorphism from a group $G$ 
to $\Aut(X)$, to ideal properties of the skew group algebra
$C(X) \rtimes^{\sigma} G$.
For more details concerning
skew group algebras, and, more generally,
skew category algebras, see Section \ref{skewcategoryalgebras}.

\begin{thm}[Öinert \cite{oinert09} and \cite{oinertarxiv11}]\label{maintheorem}
Suppose that $X$ is a compact
Hausdorff 
space and $s : G \rightarrow \Aut(X)$
is a group homomorphism.
Consider the following three assertions:
\begin{itemize}

\item[{\rm (i)}] $C(X) \rtimes^{\sigma} G$ is simple;

\item[{\rm (ii)}] $s$ is minimal;

\item[{\rm (iii)}] $s$ is faithful.

\end{itemize}
The following conclusions hold:
\begin{itemize}

\item[{\rm (a)}] {\rm (i)} implies {\rm (ii)} and {\rm (iii)};

\item[{\rm (b)}] if $G$ is an abelian group,
then {\rm (i)} holds if and only if {\rm (ii)} and {\rm (iii)} hold;

\item[{\rm (c)}] if $G$ is isomorphic to the additive group
of integers and $X$ has infinite cardinality,
then {\rm (i)} holds if and only if {\rm (ii)} holds.

\end{itemize}
\end{thm}

For related results concerning the ideal structure in
skew group algebras and, more generally, group graded rings
and Ore extensions, see
\cite{oinricsil11}, \cite{oin06}
\cite{oin07}, \cite{oin08}, \cite{oin10},
\cite{svesildejeu07}, \cite{svesildejeu08} 
and \cite{svesildejeu09}.

A natural question is if there
is a version of Theorem \ref{maintheorem} that holds
for dynamics defined by families of
{\it partial} functions on a space, that is
functions that do not necessarily
have the same domain or codomain.
In this article we address this question by
using the machinery developed by the authors 
in \cite{pajo11} for {\it category dynamical systems}.
These are defined by families, stable under composition,
of continuous maps between potentially
different topological spaces.
We show a generalization of Theorem \ref{maintheorem}
to the case of skew category algebras defined by these maps
and spaces (see Theorem \ref{genmaintheorem} below).
To be more precise, suppose that $G$ is a category.
The family of objects of $G$ is denoted by $\ob(G)$;
we will often identify an object in $G$ with
its associated identity morphism.
The family of morphisms in $G$ is denoted by $\mor(G)$;
by abuse of notation, we will often write $n \in G$
when we mean $n \in \mor(G)$.
Throughout the article $G$ is assumed to be small, that is
with the property that $\mor(G)$ is a set.
The domain and codomain of a morphism $n$ in $G$ is denoted by
$d(n)$ and $c(n)$ respectively.
We let $G^{(2)}$ denote the collection of composable
pairs of morphisms in $G$, that is all $(m,n)$ in
$\mor(G) \times \mor(G)$ satisfying $d(m)=c(n)$.
For each $e \in \ob(G)$, let $G_e$ denote the
collection of $n \in \mor(G)$ with $d(n)=c(n)=e$.
Let $G^{\op}$ denote the opposite category of $G$.
We let $\Top$ denote the category
having topological spaces as objects
and continuous functions as morphisms.
Suppose that $s : G \rightarrow \Top^{\op}$ is a (covariant) functor;
in that case we say that $s$ is a {\it category dynamical system}.
If $G$ is a groupoid, that is a category
where all morphisms are isomorphisms, then
we say that $s$ is a 
{\it groupoid dynamical system}\footnote{The notion {\it groupoid dynamical system} 
is in the $C^*$-algebra literature used
in a sense different from ours, see e.g. 
\cite{masuda1} and \cite{masuda2}.}.
If $e \in \ob(G)$, then
we say that an element $x \in s(e)$ is \emph{periodic}
if there is a nonidentity $n : e \rightarrow e$ in $G$
such that $s(n)(x)=x$;
an element of $s(e)$ which is not periodic
is called \emph{aperiodic}.
We say that $s$ is {\it topologically free} if for each $e \in \ob(G)$,
the set of aperiodic elements of $s(e)$ is dense in $s(e)$.
We say that $s$ is {\it minimal} if for each $e \in \ob(G)$,
there is no nonempty proper closed subset $Y$ 
of $s(e)$ such that $s(n)(Y) \subseteq Y$ for all 
$n \in G_e$. 
We say that $s$ is {\it faithful} if for each $e \in \ob(G)$,
and each nonidentity $n \in G_e$, there is
$x \in s(e)$ such that $s(n)(x) \neq x$.
For each $e \in \ob(G)$, we
let $C(e)$ denote the set of continuous complex valued
functions on $s(e)$.
For each $n \in G$ the functor $s$ induces a map
$\sigma(n) : C(d(n)) \rightarrow C(c(n))$
by the relation $\sigma(n)(f) = f \circ s(n)$,
for $f \in C(d(n))$.
If we use the terminology introduced in Section \ref{skewcategoryalgebras},
then the map $\sigma$ defines a {\it skew category system}
(see Definition \ref{defskewcategorysystem}).
In the same section,
we show how one to
each skew category system may associate a so called
{\it skew category algebra} $A \rtimes^{\sigma} G$ (see also \cite{oinlun08}
for a more general construction)
where $A = \oplus_{e \in \ob(G)} C(e)$.
In Section \ref{skewcategoryalgebras},
we obtain two results concerning simplicity
of a general skew category algebra $A \rtimes^{\sigma} G$.
(see Proposition \ref{simplicity} and
Proposition \ref{simplicitycommutativeA}).
These results are applied to
category dynamical systems in Section \ref{catdynsyst},
where we show the following generalization
of Theorem \ref{maintheorem} from groups to categories.

\begin{thm}\label{genmaintheorem}
Let $s : G \rightarrow \Top^{\op}$ be a category
dynamical system with the property that 
for each $e \in \ob(G)$, the 
space $s(e)$
is compact 
Hausdorff.
Consider the following four assertions:
\begin{itemize}

\item[{\rm (i)}] $\left[ \oplus_{e \in \ob(G)} C(e) \right] \rtimes^{\sigma} G$ is simple;

\item[{\rm (ii)}] $G$ is inverse connected;

\item[{\rm (iii)}] $s$ is minimal;

\item[{\rm (iv)}] $s$ is faithful.

\end{itemize}
The following conclusions hold:
\begin{itemize}

\item[{\rm (a)}] {\rm (i)} implies {\rm (ii)}, {\rm (iii)} and {\rm (iv)};

\item[{\rm (b)}] if $G$ is a locally abelian groupoid,
then {\rm (i)} holds if and only if {\rm (ii)}, {\rm (iii)}
and {\rm (iv)} hold;

\item[{\rm (c)}] if $G$ is a groupoid with the property that
for each $e \in \ob(G)$ the 
space $s(e)$
is infinite and the group $G_e$ is isomorphic to
${\Bbb Z}$, then {\rm (i)} holds if and only 
if {\rm (ii)} and {\rm (iii)} hold.

\end{itemize}
\end{thm}

Various types of crossed product algebras associated to groupoid dynamical
systems, and even more general types of dynamical systems, have appeared
in the literature before (see e.g.
\cite{exelvershik06,masuda1,masuda2,paterson99}).
The difference between these algebras and our algebras is that, generally
speaking, our skew category algebras are defined in an algebraic way,
without making use of any topology.
For example, by choosing $A$ to be a $C^*$-algebra and our category $G$ to
be a locally compact group acting on $A$, we may form the {\emph skew
category algebra} $A \rtimes^{\sigma} G$, and also the standard crossed
product $C^*$-algebra appearing in e.g. \cite{takesaki03}. The relation
between these two algebras is that the skew category algebra sits as a
dense subalgebra inside the crossed product $C^*$-algebra.

\section{Simple Skew Category Algebras}\label{skewcategoryalgebras}

In this section, we first recall the definitions
of skew category systems $(A,G,\sigma)$
and skew category algebras $A \rtimes^{\sigma} G$ from \cite{pajo11}
(see Definition \ref{defskewcategorysystem}
and Definition \ref{defskewcategoryalgebra}).
Thereafter, we 
show two results concerning simplicity
of skew category algebras and properties
of skew category systems
(see Proposition \ref{simplicity} and
Proposition \ref{simplicitycommutativeA}).
These results will be applied to
category dynamical systems in Section \ref{catdynsyst}.

\begin{conventionsrings}
Let $R$ be an associative ring.
The identity map $R \rightarrow R$ is denoted by $\id_R$.
If $R$ is unital then the identity element of $R$
is nonzero and is denoted by $1_R$.
The category of unital rings is denoted by $\Ring$.
We say that a subset $R'$ of $R$ is a subring of $R$
if it is itself a ring under the binary operations of $R$.
We always assume that ring homomorphisms between
unital rings respect the identity elements.
If $A$ is a subset of $R$, then the {\it commutant} of $A$
in $R$ is the set of elements of $R$ that commute
with every element of $A$.
If $A$ is a commutative subring of $R$, then
$A$ is called {\it maximal commutative} in $R$
if the commutant of $A$ in $R$ equals $A$.
All ideals of rings are supposed to be two-sided.
By a \emph{nontrivial ideal} we mean a proper nonzero ideal.
If $R$ is commutative and $x \in R$, then $\Ann(x)$
denotes the ideal of $R$ consisting of all $y \in R$
satisfying $xy = 0$.
If $G$ is a monoid of endomorphisms of a ring $A$,
then we say that a subset $B$ of $A$ is 
\emph{$G$-invariant} if for every $g \in G$
the inclusion $g(B) \subseteq B$ holds.
The ring $A$ is called \emph{$G$-simple} if there 
is no nontrivial $G$-invariant ideal of $A$.
\end{conventionsrings}

\begin{conventionscategories}
Let $G$ be a category.
Recall that $G$ is called {\it connected}
if its underlying undirected graph is connected.
Note that if $G$ is a groupoid, then $G$ is 
connected precisely when there to each pair $e,f \in \ob(G)$,
is $n \in \mor(G)$ with $d(n)=e$ and $c(n)=f$.
We say that $G$ is {\it locally a group}
({\it locally abelian}) if
each monoid $G_e$, for $e \in \ob(G)$, is a group (abelian).
We say that $G$ is {\it inverse connected} if given
$e,f  \in \ob(G)$, there are $m,n \in \mor(G)$
with $d(m)=c(n)=f$ and $mn = e$.
Note that if $G$ is both inverse connected
and locally a group, then $G$ is a groupoid.
A congruence relation $R$ on $G$ is a collection of
equivalence relations $R_{e,f}$ on $\hom(e,f)$,
for $e,f \in \ob(G)$, chosen so that if 
$(m,m') \in R_{e,f}$ and $(n,n') \in R_{f,g}$, 
then $(mn,m'n') \in R_{e,g}$, for all $e,f,g \in \ob(G)$. 
If $e,f \in \ob(G)$ and $n \in \hom(e,f)$,
then we let $[n]$ denote the equivalence class
in $\hom(e,f)$ defined by $R_{e,f}$. 
Suppose that $H$ is another
category and that $F : G \rightarrow H$ is a functor. 
The kernel of $F$, denoted $\ker(F)$, is the congruence relation on $G$ defined by letting 
$(m,n) \in \ker(F)_{e,f}$, for $e,f \in \ob(G)$, 
whenever $m,n \in \hom(e,f)$ and $F(m) = F(n)$.
Recall that $\ker(F)$ is called {\it trivial}
if for each $e,f \in \ob(G)$,
$\ker(F)_{e,f}$ is the equality relation on $\hom(e,f)$.
We say that $\ker(F)$ is {\it locally trivial}
if for each $e \in \ob(G)$, 
$\ker(F)_{e,e}$ is the equality relation on $\hom(e,e)$.
\end{conventionscategories}

\begin{defn}\label{defskewcategorysystem}
By a \emph{skew category system} we mean a
triple $(A,G,\sigma)$ where $G$ is a (small) category,
$A$ is the direct sum of unital rings $A_e$, for $e \in \ob(G)$,
and $\sigma$ is a functor $G \rightarrow \Ring$ satisfying
$\sigma(n) : A_{d(n)} \rightarrow A_{c(n)}$, for $n \in G$.
\end{defn}

\begin{rem}\label{intermsofmaps}
Suppose that $(A,G,\sigma)$ is a skew category system.
The fact that $\sigma$ is a functor $G \rightarrow \Ring$
can be formulated in terms of maps by saying that
\begin{equation}\label{idd}
\sigma(e) = {\rm id}_{A_e}
\end{equation}
for all $e \in \ob(G)$, and
\begin{equation}\label{algebraa}
\sigma(m) \sigma(n) = \sigma(mn)
\end{equation}
for all $(m,n) \in G^{(2)}$.
\end{rem}

\begin{defn}\label{defskewcategoryalgebra}
If $(A,G,\sigma)$ is a skew category system, then
we let $A \rtimes^{\sigma} G$ denote the collection of formal sums
$\sum_{n \in G} a_n u_n$, where $a_n \in A_{c(n)}$, $n \in G$,
are chosen so that all but finitely many of them are nonzero.
Define addition and multiplication on $A \rtimes^{\sigma} G$ by
\begin{equation}\label{addition}
\left( \sum_{n \in G} a_n u_n \right) +
\left( \sum_{n \in G} b_n u_n \right) =
\sum_{n \in G} \left( a_n + b_n \right) u_n
\end{equation}
respectively
\begin{equation}\label{multiplication}
\left( \sum_{n \in G} a_n u_n \right)
\left( \sum_{n \in G} b_n u_n \right) =
\sum_{n \in G} \left( \sum_{\stackrel{(m,m') \in G^{(2)};}{mm' = n}}
a_m \sigma(m)(b_{m'}) \right) u_n
\end{equation}
for $\sum_{n \in G} a_n u_n$, $\sum_{n \in G} b_n u_n \in A \rtimes^{\sigma} G$.
It is clear that these operations define
a ring structure on $A \rtimes^{\sigma} G$.
We call $A \rtimes^{\sigma} G$ the \emph{skew category algebra}
defined by $(A,G,\sigma)$.
Often we let $u_n$ denote $1_{A_{c(n)}} u_n$
for all $n \in G$.
\end{defn}

\begin{rem}
If $G$ is a groupoid, then (\ref{multiplication})
can be rewritten in the following slightly simpler form
\begin{equation*}\label{groupoidproduct}
\left( \sum_{n \in G} a_n u_n \right)
\left( \sum_{n \in G} b_n u_n \right) =
\sum_{n \in G} \left( \sum_{\stackrel{m \in G;}{c(m)=c(n)}}
a_m \sigma(m)(b_{m^{-1}n}) \right) u_n,
\end{equation*}
which, in the case when $G$ is a group, 
simplifies even more to
\begin{equation*}\label{groupproduct}
\left( \sum_{n \in G} a_n u_n \right)
\left( \sum_{n \in G} b_n u_n \right) =
\sum_{n \in G} \left( \sum_{m \in G}
a_m \sigma(m)(b_{m^{-1}n}) \right) u_n.
\end{equation*}
\end{rem}

\begin{rem}
Suppose that $T := A \rtimes^{\sigma} G$ is a skew category algebra.
If we for each $n \in G$, put $T_n = A_{c(n)} u_n$,
then $T = \oplus_{n \in G} T_n$,
$T_m T_n = T_{mn}$, for $(m,n) \in G^{(2)}$,
and $T_m T_n = \{ 0 \}$, otherwise.
In the terminology of \cite{liu06}, \cite{lu06}, \cite{oinlun10}
and \cite{oinlunMiyashita}
this means that a skew category algebra is a strongly category
graded ring.
\end{rem}

\begin{prop}\label{intersection}
If $A \rtimes^{\sigma} G$ is a skew
category algebra and $I$ is an ideal of
$A \rtimes^{\sigma} G$, then
\begin{itemize}

\item[{\rm (a)}] the equality $I = A \rtimes^{\sigma} G$
holds if and only the equality $I \cap A_e = A_e$ holds
for all $e \in \ob(G)$;

\item[{\rm (b)}] if $G$ is inverse connected,
then the equality $I = A \rtimes^{\sigma} G$ holds
if and only if the equality $I \cap A_e = A_e$ holds
for some $e \in \ob(G)$.
\end{itemize}
\end{prop}

\begin{proof}
Let $R$ denote $A \rtimes^{\sigma} G$.

(a) The ''only if'' statement is clear.
Now we show the ''if'' statement.
Suppose that for every $e \in \ob(G)$
the equality $I \cap A_e = A_e$ holds.
Take $x \in R$.
From the definition of skew category
algebras it follows that there is
a finite subset $X$ of $\ob(G)$
satisfying
\begin{equation}\label{identity}
x \left( \sum_{e \in X}u_e \right) = x.
\end{equation}
But $\sum_{e \in X} u_e \in \sum_{e \in X} A_e
= \sum_{e \in X} I \cap A_e \subseteq I$ which,
by equation (\ref{identity}),
implies that $x \in I$.
Since $x$ was arbitrarily chosen from $R$,
we get that $I = R$.

(b) The ''only if'' statement is clear.
Now we show the ''if'' statement.
Suppose that $I \cap A_e = A_e$ for some $e \in \ob(G)$.
Take $f \in \ob(G)$. By (a) we are done
if we can show that $I \cap A_f = A_f$.
Since $G$ is inverse connected, there are
$m,n \in G$ with $d(m)=c(n)=e$
such that $f = mn$.
But then $u_f = u_m u_e u_n \in
I \cap A_f$. This implies that $I \cap A_f = A_f$.
\end{proof}

\begin{prop}\label{estonia}
Let $A \rtimes^{\sigma} G$ be a skew
category algebra.
Suppose that $R$ is a congruence relation
contained in $\ker(\sigma)$. 
If I is the two-sided ideal in $A \rtimes^{\sigma} G$ 
generated by an element $\sum_{n \in \mor(G)} a_n u_n$, 
where $a_n \in A_{c(n)}$, for $n \in \mor(G)$,
with $a_n = 0$ for all but finitely many $n \in \mor(G)$,
satisfying $a_n = 0$ if $n$ does not belong to any of the 
classes $[e]$, for $e \in \ob(G)$, and
$\sum_{n \in [e]} a_n = 0$, for $e\in \ob(G)$, 
then $A \cap I = \{ 0 \}$.
In particular, if $A \rtimes^{\sigma} G$ is simple,
then $\ker(\sigma)$ is locally trivial.
\end{prop}

\begin{proof}
See the proof of \cite[Proposition 9]{oinlun08}.
\end{proof}

\begin{prop}\label{simplicity}
Let $A \rtimes^{\sigma} G$ be a skew
category algebra.  
Consider the following five assertions:
\begin{enumerate}[{\rm (i)}]
    \item $A \rtimes^{\sigma} G$ is simple;
    \item $G$ is inverse connected;
  \item for each $e \in \ob(G)$, the ring $A_e$ is $G_e$-simple;  
	\item for each $e \in \ob(G)$, the ring
	$Z(A_e \rtimes^{\sigma} G_e)$ is a field;
	\item $\ker(\sigma)$ is locally trivial.
\end{enumerate}
The following conclusions hold:
\begin{enumerate}[{\rm (a)}]
\item {\rm (i)} implies {\rm (ii)}-{\rm (v)};

\item if $G$ is a locally abelian groupoid, then 
{\rm (i)} holds if and only if {\rm (ii)}-{\rm (v)} hold.

\end{enumerate}
\end{prop}

\begin{proof}
Let $R$ denote $A \rtimes^{\sigma} G$.

(a) Suppose that (i) holds.
We first show (ii).
Take $e,f \in \ob(G)$. Let $I$ denote the ideal 
generated by $u_e$ in $R$.
Since $R$ is simple it follows that $I = R$.
In particular, we get that $u_f \in I$.
Since $I$ consists of the set of finite sums
of the form $x u_e y$, where $x,y \in R$,
it follows that there exist $m,n \in \mor(G)$
with $d(m)=c(n)=e$ such that $u_f = u_m u_e u_n$.
This implies that $f = mn$ and hence that $G$ 
is inverse connected.
Now we show (iii). Take $e \in \ob(G)$
and a nonzero $G_e$-invariant ideal $J_e$ of $A_e$.
Let $I$ denote the ideal of $R$ generated by
$J_e u_e$. Since $R$ is simple we get that $I = R$.
This implies in particular that
$u_e \in I$. 
Then $u_e \in I \cap A_e u_e = J_e u_e$
which implies that 
$1_{A_e} \in J_e$. Hence $J_e = A_e$.
Now we show (iv).
Let $e\in \ob(G)$.
Take a nonzero $x$
in $Z(A_e \rtimes^{\sigma} G_e)$ 
and let $I$ be the ideal of $R$ generated by $x$.
Since $I$ is nonzero and $R$ is simple, we get that $I=R$.
In particular, $u_e$ equals a finite
sum of elements of the form $yxz$
where $y,x \in A_e \rtimes^{\sigma} G_e$.
But since $x$ belongs to $Z(A_e \rtimes^{\sigma} G_e)$
we get that $u_e = w x = xw$ for some 
$w \in A_e \rtimes^{\sigma} G_e$. 
All that is left to show now is that
$w \in Z(A_e \rtimes^{\sigma} G_e)$.
Take $v \in A_e \rtimes^{\sigma} G_e$.
Then, since $x$ commutes with $v$, we get that
$w v = w v u_e = w v x w = w x v w = u_e v w = v w$.
Assertion (v) follows immediately from Proposition \ref{estonia}.

(b) From (a) it follows that we only need to
show the ''if'' statement.
Suppose that (i)-(v) hold.
Let $I$ be a nonzero ideal of $R$.
Then there is a nonzero element 
$x = \sum_{n \in \mor(G)} a_n u_n$ in $I$,
where $a_n \in A_{c(n)}$, for $n \in \mor(G)$ 
and $a_n = 0$ for all but finitely many $n \in \mor(G)$,
with the property that $a_e \neq 0$
for some $e \in \ob(G)$.
Indeed, take a nonzero
$y = \sum_{m \in \mor(G)} b_m u_m$ in $I$.
We now consider two cases.
Case 1: There is $n \in \mor(G)$
with $d(n)=c(n)$ and $b_n \neq 0$.
Then $x = yu_{n^{-1}}$ has the desired property 
where $e = c(n)$. 
Case 2: There is $n \in \mor(G)$ with 
$d(n) \neq c(n)$ and $b_n \neq 0$.
Since $G$ is inverse connected it
follows that there are $m,p \in \mor(G)$
with $d(m) = c(n)$, $c(p)=d(p)$ and $mp=c(n)$.
Then $x = yu_{(nm)^{-1}}$ has the desired property
where $e = c(n)$.
Let $J$ be the ideal of $A_e \rtimes^{\sigma} G_e$
consisting of all $b \in A_e$ such that there
are $b_n \in A_e$, for $n \in G_e \setminus \{ e \}$,
with the property that $b_n \in A_e$, for 
$n \in G_e \setminus \{ e \}$, where
$b_n = 0$ for all but finitely many
$n \in G_e \setminus \{ e \}$, and
$y := b + \sum_{n \in G_e \setminus \{ e \}} b_n u_n \in u_e I u_e$
and $\Supp(y) \subseteq \Supp(x)$.
By the above discussion concerning the element
$x$ it follows that $J$ is nonzero.
Now we show that $J$ is $G_e$-invariant.
Take $m \in G_e$. Then, since $G_e$ is abelian,
it follows that 
$$I \ni u_m y u_{m^{-1}} =
u_m b u_{m^{-1}} + 
\sum_{m \in G_e \setminus \{ e \}} u_m b u_n u_{m^{-1}} =
\sigma(m)(b) + \sum_{n \in G_e \setminus \{ e \}} \sigma(m)(b_n) u_n.$$
Therefore, $\sigma(m)(b) \in J$. 
By $G_e$-simplicity of $A_e$ it follows that $J = A_e$.
In particular, we can choose $y$ so that $b = 1_{A_e}$.
Among all nonzero elements 
$z = \sum_{n \in G_e} c_n u_n \in u_e I u_e$,
with $c_n = 0$ for all but finitely many $n \in G_e$,
choose an element minimizing $|\Supp(z)|$. 
By the above discussion, we can assume that $c_e=1$
for such an element $z$. 
Now we show that $z \in Z(A_e \rtimes^{\sigma} G_e)$.
Take $a \in A_e$ and $m \in G_e$.
Then, since $G_e$ is abelian, it follows that
$$a u_m z - z a u_m = 
\sum_{n \in G_e} \left( a \sigma(m)(c_n) - c_n \sigma(n)(a) \right) u_{mn}.$$
Since $$a \sigma(m) (c_e) - c_e \sigma(e)(a) =
a \sigma(m) (1_{A_e}) - 1_{A_e} a = 
a 1_{A_e} - 1_{A_e} a = a - a = 0$$
we get that $|\Supp(a u_m z - z a u_m)| < |\Supp(z)|$.
Since $a u_m z - z a u_m \in u_e I u_e$, we get,
by minimality of $|\Supp(z)|$, that 
$a u_m z - z a u_m = 0$.
Therefore $z \in Z(A_e \rtimes^{\sigma} G_e)$.
Since $Z(A_e \rtimes^{\sigma} G_e)$ is a field
and $z$ is nonzero, it follows that 
$z$ is invertible in $A_e \rtimes^{\sigma} G_e$
and hence that $u_e I u_e = A_e \rtimes^{\sigma} G_e$.
In particular $I \cap A_e = A_e$.
Simplicity of $A \rtimes^{\sigma} G$
now follows directly from  Proposition \ref{intersection}(b).
\end{proof}

\begin{prop}\label{eqconditions}
Let $A \rtimes^{\sigma} G$ be a skew
category algebra with $A$ commutative. Consider the following two assertions:
\begin{enumerate}[{\rm (i)}]
\item if $I$ is a nonzero ideal of $A \rtimes^{\sigma} G$,
then $I \cap A \neq \{ 0 \}$;

\item the subring $A$ is maximal commutative in $A \rtimes^{\sigma} G$.
\end{enumerate}
The following conclusions hold:
\begin{enumerate}[{\rm (a)}]
\item {\rm (i)} implies {\rm (ii)};

\item {\rm (ii)} does not imply {\rm (i)} for all categories $G$;

\item if $G$ is a groupoid,
then {\rm(i)} holds if and only if {\rm (ii)} holds.
\end{enumerate}
\end{prop}

\begin{proof}
See the proof of \cite[Proposition 2.3]{pajo11}.
\end{proof}

\begin{rem}
For other results related to the implication
(i) implies (ii) in Proposition \ref{eqconditions},
see \cite{oinlun08}.
The implication (ii) implies (i) in Proposition \ref{eqconditions}
actually holds for all nondegenerate groupoid graded rings,
see \cite{oinlun10}.
\end{rem}

\begin{prop}\label{simplicitycommutativeA}
Let $A \rtimes^{\sigma} G$ be a skew
category algebra with $A$ commutative. 
Consider the following six assertions:
\begin{enumerate}[{\rm (i)}]
    \item $A \rtimes^{\sigma} G$ is simple;
    \item $G$ is inverse connected;
  \item for each $e \in \ob(G)$, the ring $A_e$ is $G_e$-simple;  
	\item for each $e \in \ob(G)$, the ring
	$Z(A_e \rtimes^{\sigma} G_e)$ is a field;
	\item $\ker(\sigma)$ is locally trivial.
	\item $A$ is maximal commutative in $A \rtimes^{\sigma} G$.
\end{enumerate}
The following conclusions hold:
\begin{enumerate}[{\rm (a)}]
\item {\rm (i)} implies {\rm (ii)}-{\rm (vi)};

\item if $G$ is a groupoid, 
then {\rm (i)} holds if and only if {\rm (ii)}-{\rm (vi)} hold.
\end{enumerate}
\end{prop}

\begin{proof}
Let $R$ denote $A \rtimes^{\sigma} G$.

(a) Suppose that (i) holds.
By Proposition \ref{simplicity}(a) we get that 
(ii)-(v) hold. It follows from Proposition \ref{eqconditions}(a)
that (vi) holds.

(b) Suppose that $G$ is a groupoid.
By (a) we only need to show the ''if'' statement.  
Suppose that (ii)-(vi) hold.
We show (i). Let $I$ be a nonzero ideal of $R$.
By Proposition \ref{eqconditions}(b),
we get that $I \cap A$ is a nonzero ideal of $A$.
Since $I \cap A = \sum_{e \in \ob(G)} I \cap A_e$,
there is $e \in \ob(G)$ such that
$I \cap A_e$ is a nonzero ideal of $A_e$.
Take $n \in G_e$. From the fact that 
$G$ is a groupoid, we get that
there is $m \in G_e$ such that $nm = e$.
Then $$ n(I \cap A_e) u_e = 
\sigma(n)(I \cap A_e) u_{nm} =
\sigma(n)(I \cap A_e) u_n u_m =
u_n (I \cap A_e) u_m \subseteq $$
$$ \subseteq (u_n I u_m) \cap u_n A_e u_m \subseteq
I \cap \sigma(n)(A_e) u_n u_m \subseteq
(I \cap A_e)u_{nm} =
(I \cap A_e)u_e.$$
Hence $u_n(I \cap A_e) \subseteq I \cap A_e$ and
thus $I \cap A_e$ is also $G_e$-invariant.
By $G_e$-simplicity of
$A_e$ this implies that $I \cap A_e = A_e$.
By Proposition \ref{intersection}(b), we get that $I=R$.
\end{proof}

\section{Minimal and Faithful Category Dynamical Systems}\label{catdynsyst}

In this section, we prove Theorem \ref{genmaintheorem}.
To this end, we need a result from \cite{pajo11} 
(see Proposition \ref{equivalentpajo11})
concerning topological freeness and maximal commutativity.
We shall also need three results (see Proposition \ref{minimalimpliesweaklytopfree},
Proposition \ref{equivalentminimal1} and
Proposition \ref{equivalentminimal2})
relating minimality of category dynamical systems
to simplicity of the corresponding
skew category algebras.
In the end of this section, we discuss the implications
of these results for the connection between
dynamical systems defined by partially defined functions
(see Definition \ref{partialdefn}) and properties
of the corresponding skew category algebras
(see Examples \ref{partialexmp} and \ref{examplediscrete}).

Let $s : G \rightarrow \Top^{\op}$ be a category dynamical system.
Then $(\oplus_{e \in \ob(G)}C(e) , G , \sigma)$
is a skew category system.
Indeed, we need to check conditions (\ref{idd})
and (\ref{algebraa}) from Remark \ref{intermsofmaps}.
Take $e \in \ob(G)$ and $f \in C(e)$.
Then $\sigma(e)(f) = f \circ s(e) = f.$
Therefore $\sigma(e) = \id_{C(e)}$.
Take $(m,n) \in G^{(2)}$ and $f \in C(d(n))$.
Then $\sigma(m)\sigma(n)(f) = \sigma(m)( f \circ s(n) )=
f \circ s(n) \circ s(m) = f \circ (s(m) \circ_{\op} s(n))=
f \circ s(mn) = \sigma(mn)(f).$
Therefore $\sigma(m) \sigma(n) = \sigma(mn)$.
Hence, we may form the skew category algebra
$(\oplus_{e \in \ob(G)} C(e)) \rtimes^\sigma G$.

\begin{prop}\label{equivalentpajo11}
Suppose that $s : G \rightarrow \Top^{\op}$
is a groupoid dynamical system.
If for each $e \in \ob(G)$ the 
space
$s(e)$ is locally compact Hausdorff,
then the following two assertions are equivalent:
\begin{enumerate}[{\rm (i)}]

\item $s$ is topologically free;

\item the subring $\oplus_{e \in \ob(G)} C(e)$ 
is maximal commutative in
$[\oplus_{e \in \ob(G)} C(e)] \rtimes^{\sigma} G$.

\end{enumerate}
\end{prop}

\begin{proof}
See the proof of \cite[Theorem 3.2]{pajo11}.
\end{proof}

\begin{lem}\label{image}
Suppose that $X$ and $Y$ are topological spaces
and $A \subseteq X$ and $B \subseteq Y$.
If $f : X \rightarrow Y$ is a continuous function
such that $f(A) \subseteq B$,
then $f \left( \overline{A} \right) \subseteq \overline{B}$.
\end{lem}

\begin{proof}
The inclusion $f(A) \subseteq B$ can equivalently be
stated as the inclusion $A \subseteq f^{-1}(B)$.
Since $f$ is continuous we get that
$f^{-1} \left( \overline{B} \right)$ is a closed subset
of $X$ containing $A$ and hence also $\overline{A}$,
i.e.
$\overline{A} \subseteq f^{-1} \left( \overline{B} \right)$
or equivalently
$f \left( \overline{A} \right) \subseteq \overline{B}$.
\end{proof}

\begin{prop}\label{minimalimpliesweaklytopfree}
If $s : G \rightarrow \Top^{\op}$ is a minimal
category dynamical system such that for each
$e \in \ob(G)$ the topological space $s(e)$ is
infinite and Hausdorff and $G_e$ is isomorphic
to the additive group of integers, 
then $s$ is topologically free and,
hence, faithful.
\end{prop}

\begin{proof}
Take $e \in  \ob(G)$.
We will show that \emph{every} point in $s(e)$ is aperiodic,
which, of course, implies that
the set of aperiodic points in $s(e)$ is dense.
Take $x \in s(e)$. Then $G_e(x)$ is a nonempty
$G_e$-invariant subset of $s(e)$.
By Lemma \ref{image} the set $\overline{G_e(x)}$ is a
nonempty closed $G_e$-invariant
subset of $X$, which, since $s$ is minimal, implies
that $\overline{G_e(x)} = s(e)$.
Since $s(e)$ is infinite and Hausdorff it follows that
$G_e(x)$ is infinite.
Let $g$ be a generator for $G_e$.
Seeking a contradiction, suppose that
there is a nonzero integer $N$ such that $s(g^N)(x)=x$.
Then the cardinality of $G_e(x)$ is less than
or equal to $|N|$. This contradicts the fact that
$G_e(x)$ is an infinite set.
Therefore, $s(g^N)(x) \neq x$, for all nonzero integers
$N$, and hence $x$ is an aperiodic point in $s(e)$.
\end{proof}

Recall that a topological space $X$ is called
\emph{completely regular} if given any closed 
proper subset $F$ of $X$,
there is a nonzero continuous complex valued function 
on $X$ that vanishes on $F$.

\begin{lem}\label{compactlemma}
Every compact Hausdorff topological space
is completely regular.
\end{lem}

\begin{proof}
See any standard book on point set topology, e.g. \cite{mun00}.
\end{proof}

\begin{prop}\label{equivalentminimal1}
Suppose that $s : G \rightarrow \Top^{\op}$ is a category
dynamical system with each $s(e)$, for $e \in \ob(G)$,
compact Hausdorff.
If for each $e \in \ob(G)$, the ring $C(e)$
is $G_e$-simple, then $s$ is minimal.
\end{prop}

\begin{proof}
We show the contrapositive statement.
Suppose that $s$ is not minimal.
Then there is $e \in \ob(G)$ such that
$s(e)$ is not $G_e$-minimal, that is there is a
closed nonempty proper $G_e$-invariant subset
$Y$ of $s(e)$. Define $I_Y$ to be the set of $f \in C(e)$
that vanish on $Y$.
It is clear that $I_Y$ is a $G_e$-invariant proper ideal of $C(e)$.
By Lemma \ref{compactlemma}, it follows that
$I_Y$ is also nonzero. Therefore,
$C(e)$ is not $G_e$-simple.
\end{proof}

\begin{prop}\label{equivalentminimal2}
Suppose that $s : G \rightarrow \Top^{\op}$ is a category
dynamical system with each $s(e)$, for $e \in \ob(G)$,
compact.
If $s$ is minimal, then 
for each $e \in \ob(G)$, the ring $C(e)$
is $G_e$-simple.
\end{prop}

\begin{proof}
We show the contrapositive statement.
Suppose that there is $e \in \ob(G)$
such that $C(e)$ is not $G_e$-simple.
Then there is a nontrivial $G_e$-invariant
ideal $I$ of $C(e)$.
For a subset $J$ of $C(e)$,
let $N_J$ denote the set 
$\bigcap_{f \in J} f^{-1}( \{ 0 \} )$.
We claim that $N_I$ is a closed, nonempty
proper $G_e$-invariant subset of $C(e)$.
If we assume that the claim holds,
then $s$ is not minimal and the proof is done.
Now we show the claim.
Since $I$ is $G_e$-invariant
the same is true for $N_I$.
Since $I$ is nonzero it follows that $N_I$
is a proper subset of $s(e)$.
Since each set $f^{-1}( \{ 0 \} )$, for $f \in I$,
is closed, the same is true for $N_I$.
Seeking a contradiction, suppose that $N_I$ is empty.
Since $C(e)$ is compact,
there is a finite subset $J$ of $I$ such that
$N_J = N_I$. Then the function
$F = \sum_{f \in J} |f|^2 =
\sum_{f \in J} f \overline{f}$ belongs to $I$
and, since $N_J$ is empty, 
it has the property that $F(x) \neq 0$ for all $x \in s(e)$.
Therefore $1_e = F \cdot \frac{1}{F} \in I$,
where $1_e$ denotes the constant map 
$s(e) \rightarrow {\Bbb C}$ which
sends each element of $s(e)$ to 1. 
This implies that $I = C(e)$
which is a contradiction.
Therefore $N_I$ is nonempty.
\end{proof}

\begin{prop}\label{faithfullocallytrivial}
If $s : G \rightarrow \Top^{\op}$ is a category
dynamical system with each $s(e)$, for $e \in \ob(G)$,
compact Hausdorff, then $s$ is faithful if and only 
if $\ker(\sigma)$ is locally trivial.
\end{prop}

\begin{proof}
Suppose that $s$ is not faithful.
Then there is $e \in \ob(G)$ and a 
nonidentity $n \in G_e$ such that 
$s(n) = \id_{s(e)}$. This implies that
$\sigma(n) = \id_{C(e)}$ and hence that
$\sigma$ is not locally trivial.

Suppose that $\sigma$ is not locally trivial.
Then there is $e \in \ob(G)$ and a nonidentity
$n \in G_e$ such that $\sigma(n) = \id_{C(e)}$.
This implies that $f(s(n)(x)) = f(x)$ for all
$f \in C(e)$ and all $x \in s(e)$.
By Urysohn's lemma (see any standard book 
on point set topology, e.g. \cite{mun00}), 
the set of continuous complex valued
functions on a compact Hausdorff
space separates points. Hence, we get that
$s(n)(x) = x$ for all $x \in s(e)$.
Therefore $s$ is not faithful.
\end{proof}

\noindent {\bf Proof of Theorem \ref{genmaintheorem}.}
Let $R$ denote  $[ \oplus_{e \in \ob(G)} C(e) ] \rtimes^\sigma  G$.

(a) Suppose that $R$ is simple.
By Proposition \ref{simplicity}(a) it follows that 
$G$ is inverse connected and that
each $C(e)$, for $e \in \ob(G)$, is $G_e$-simple.
By Proposition \ref{equivalentminimal1}, 
this implies that $s$ is minimal.
By Proposition \ref{simplicity}(a) 
again it follows that $\ker(\sigma)$
is locally trivial.
By Proposition \ref{faithfullocallytrivial}
we get that $s$ is faithful.

(b) By (a) we only need to show the ''if'' statement.
Suppose that $G$ is a locally abelian inverse connected groupoid
and that $s$ is minimal and faithful.
We show that $R$ is simple.
By Proposition \ref{equivalentminimal2} it follows
that each $C(e)$, for $e \in \ob(G)$,
is $G_e$-simple
and by Proposition \ref{faithfullocallytrivial} we conclude that $\ker(\sigma)$ is locally trivial.
Take $e \in \ob(G)$. 
We claim that $Z(C(e) \rtimes^{\sigma} G_e)$
is a field. If we assume that the claim holds,
then, by Proposition \ref{simplicity}(b), 
it follows that $R$ is simple.
Now we show the claim.
We will in fact show that 
$Z(C(e) \rtimes^{\sigma} G_e) = {\Bbb C}$
where we identify a complex number $z$
with the constant function $1_z$
in $C(e)$ that maps each element of $s(e)$ to $z$.
It is clear that 
$Z(C(e) \rtimes^{\sigma} G_e) \supseteq {\Bbb C}$.
Now we show the inclusion
$Z(C(e) \rtimes^{\sigma} G_e) \subseteq {\Bbb C}$.
Take $\sum_{n \in G_e} f_n u_n \in Z(C(e) \rtimes^{\sigma} G_e)$ 
where $f_n \in C(e)$, for $n \in G(e)$,
and $f_n = 0$ for all but finitely many $n \in G_e$.
For every $m \in G_e$, the equality
$u_m \left( \sum_{n \in G_e} f_n u_n \right) 
= \left( \sum_{n \in G_e} f_n u_n \right) u_m$
holds. From the fact that $G_e$ is abelian,
we get that $f_n(s(m)(x)) = f_n(x)$,
for $m,n \in G_e$ and $x \in s(e)$.
For every $n \in G_e$ choose a complex number
$z_n$ in the image of $f_n$.
Since $f_n \circ s(m) = f_n$ it follows that
the set $f_n^{-1}(z_n)$ is a nonempty
$G_e$-invariant closed subset of $s(e)$.
Since $s$ is minimal it follows that 
$f_n^{-1}(z_n) = s(e)$ and hence that 
$f_n = 1_{z_n}$, for $n \in G_e$.
Take a nonidentity $m \in G_e$.
From the fact that $s$ is faithful,
we get that there is $a \in s(e)$
such that $s(m)(a) \neq a$.
Since $C(e)$ separates the points in $s(e)$
there is $g \in C(e)$ such that
$g(a) \neq g(s(m)(a))$.
Since $\sum_{n \in G_e} 1_{z_n} u_n$
commutes with $g$ we get that
$1_{z_m}(x) ( g(x) - \sigma(m)(g)(x) ) = 0$
for all $x \in s(e)$. By specializing
this equality with $x=a$, we get that
$z_m (g(a) - g(s(m)(a)) = 0$
which in turn implies that $z_m = 0$.
Therefore, the inclusion 
$Z(C(e) \rtimes^{\sigma} G_e) \subseteq {\Bbb C}$ holds.

(c) We can show this in two ways.
Either, we use Theorem \ref{genmaintheorem}(b) 
and the faithful part of 
Proposition \ref{minimalimpliesweaklytopfree},
or we can construct a direct proof (similar
to the proof of (b) above) using
Proposition \ref{simplicitycommutativeA},
the topologically free part of
Proposition \ref{minimalimpliesweaklytopfree}
and Proposition \ref{equivalentpajo11}.

\hfill $\square$

\begin{rem}
If we omit the condition that $s$ is faithful,
then the conditions (ii) and (iii) 
in Theorem \ref{genmaintheorem} do not
necessarily imply that 
$[ \oplus_{e \in \ob(G)} C(e) ] \rtimes^\sigma  G$ is simple.
In fact, let $G$ and $H$ be any nontrivial groups.
By abuse of notation, we let $e$ denote
the identity element of both groups.
Suppose that $X$ is a compact Hausdorff space
equipped with a minimal $G$-action
$G \times X \ni (g,x) \mapsto g(x) \in X$.
Define an action of $G \times H$ on $X$
by the relation 
$G \times H \times X \ni (g,h,x) \mapsto g(x) \in X$;
this action is also minimal.
Then $C(X) \rtimes^{\sigma} (G \times H)$ 
is not simple.
In fact, let $I$ be the ideal generated by 
the set of elements of the form $u_{(e,e)} - u_{(e,h)}$,
for $h \in H$.
Define the homomorphism of abelian groups
$\varphi : C(X) \rtimes^{\sigma} (G \times H) \rightarrow C(X)$
by the additive extension of the relation
$\varphi( f u_{(g,h)} ) = f$, for $g \in G$, 
$h \in H$ and $f \in C(X)$.
We claim that $I \subseteq \ker(\varphi)$.
If we assume that the claim holds, then
$I$ is a nontrivial ideal of 
$C(X) \rtimes^{\sigma} (G \times H)$,
since $\varphi |_{C(X)} = {\rm id}_{C(X)}$.
Now we show the claim.
By the definition of $I$ it follows that
it is enough to show that $\varphi$
maps elements of the form
$f_1 u_{(r,s)} ( u_{(e,e)} - u_{(e,h)} ) f_2 u_{(t,v)}$ to zero,
where $f_1,f_2 \in C(X)$, $r,t \in G$ and $s,h,v \in H$.
However, since $\sigma(e,h)(f_2) = f_2$, we get that
$f_1 u_{(r,s)} ( u_{(e,e)} - u_{(e,h)} ) f_2 u_{(t,v)} =
f_1 \sigma(r,s)(f_2)( u_{(rt,sv)} - u_{(rt,shv)})$
which, obviously, is mapped to zero by $\varphi$.
\end{rem}

\begin{rem}
If we omit the condition that $G$ is locally abelian
in Theorem \ref{genmaintheorem}, then 
the conditions (ii), (iii) and (iv) do not 
necessarily imply that 
$[ \oplus_{e \in \ob(G)} C(e) ] \rtimes^\sigma  G$
is simple. 
Indeed, Öinert has given an example of this
phenomenon when $G$ is the nonabelian
group of homeomorphisms of the circle $S^1$
acting on the compact Hausdorff space $S^1$
(for the details, see \cite[Example 6.1]{oinertarxiv11}). 
\end{rem}

\begin{defn}\label{partialdefn}
Suppose that $X$ is a topological space.
By a {\it partially defined dynamical system}
on $X$ we mean a collection $P$ of functions
such that:
\begin{itemize}
\item if $f\in P$, then the domain $d(f)$ of $f$
and the codomain $c(f)$ of $f$ are subsets of $X$
and $f$ is continuous as a function
$d(f) \rightarrow c(f)$ where $d(f)$ and $c(f)$ are equipped
with the relative topologies induced by the topology on $X$;

\item if $f \in P$, then $\id_{d(f)} \in P$ and $\id_{c(f)} \in P$; 

\item if $f,g\in P$ are such that $d(f)=c(g)$, then $f \circ g\in P$.
\end{itemize}
We say that an element $x$ of $X$ is {\it periodic} with respect to $P$
if there is a nonidentity function $f$ in $P$
with $d(f)=c(f)$ and $f(x)=x$.
An element $x$ is {\it aperiodic} with respect to $P$
if it is not periodic.
We say that $P$ is \emph{topologically free} if the set
of aperiodic elements of $X$ is dense in $X$.
We say that $P$ is \emph{minimal} if for every
$Y\subseteq X$, satisfying $Y=d(g)$ for some $g\in P$,
there is no nonempty proper closed subset $S\subseteq Y$
such that $f(S)\subseteq S$ for all $f\in P$ satisfying $d(f)=c(f)=Y$.
We say that $P$ is {\it faithful} if given
a nonidentity function $f \in P$ with 
$d(f)=c(f)$, then there is $x \in d(f)$
such that $f(x) \neq x$.
By abuse of notation, we let $P$ denote the category
having the domains and codomains of functions in $P$
as objects and the functions of $P$ as morphisms.
We let the obvious functor $P \rightarrow \Top$
be denoted by $t_P$. Let $G_P$ denote the opposite
category of $P$ and let $s_P : G_P \rightarrow \Top^{\op}$
denote the opposite functor of $t_P$.
We will call $s_P$ the \emph{category dynamical system on $X$
defined by the partially defined dynamical system $P$}.
\end{defn}

\begin{prop}
If $X$ is a topological space and $P$ is a
partially defined dynamical system on $X$,
then $P$ is topologically free (minimal, faithful) 
if and only if
$s_P$ is topologically free (minimal, faithful) 
as a category dynamical system.
\end{prop}

\begin{proof}
This follows immediately from the definition
of topological freeness (minimality, faithfulness) 
of $P$ and $s_P$.
\end{proof}

To illustrate the above definitions and results,
we end the article with some concrete examples of
partially defined dynamical systems.

\begin{exmp}\label{partialexmp}
Suppose that we let $X$ denote the real numbers
equipped with its usual topology
and we let $Y$ denote the set of nonnegative real numbers equipped with
the relative topology induced by the topology on $X$.
Let $\opsqr : X \rightarrow Y$
and $\opsqrt : Y \rightarrow X$ denote the square function
and the square root function, respectively.
Furthermore, let $\opabs : X \rightarrow X$ denote the absolute value.
Let $P$ be the partially defined dynamical system
with $\ob(P) = \{ X,Y \}$ and $\mor(P) = \{ \id_X , \id_Y , \opsqr , \opsqrt , \opabs \}$.
Then we get the following table of partial composition for $P$
$$\begin{array}{c|ccccc}
\circ & \id_X & \id_Y & \opsqr & \opsqrt  & \opabs \\ \hline
\id_X & \id_X & *     & *   & \opsqrt  & \opabs \\
\id_Y & *     & \id_Y & \opsqr & *     & *   \\
\opsqr   & \opsqr   & *     & *   & \id_Y & \opsqr \\
\opsqrt  & *     & \opsqrt  & \opabs & *     & * \\
\opabs   & \opabs   & *     & *   & \opsqrt  & \opabs
\end{array}$$
Put $G = P^{\op}$ and let $A_X$ and $A_Y$
denote the set of continuous complex valued functions 
on $X$ and $Y$ respectively.
Take $f_X , f_X' , g_X , g_X' , h_X , h_X' \in A_X$ and
$f_Y , f_Y' , g_Y , g_Y' \in A_Y$. Then the product of
$$B_1 := f_X u_{\id_X} + g_X u_{\opabs} + h_X u _{\opsqr} + f_Y u_{\id_Y} + g_Y u_{\opsqrt}$$
and
$$B_2 := f_X' u_{\id_X} + g_X' u_{\opabs} + h_X' u _{\opsqr} + f_Y' u_{\id_Y} + g_Y' u_{\opsqrt}$$
in the skew category algebra $A \rtimes^{\sigma} G$ equals
$$
B_1 B_2 = f_X f_X' u_{\id_X} +
\left( f_X g_X' + g_X (f_X' \circ \opabs) + g_X (g_X' \circ \opabs) + h_X (g_Y' \circ \opsqr) \right) u_{\opabs} +$$
$$+ \left( f_X h_X' + h_X (f_Y' \circ \opsqr) + g_X (h_X' \circ \opabs) \right) u_{\opsqr} +
\left( f_Y f_Y' + g_Y (h_X' \circ \opsqrt) \right) u_{\id_Y} +$$
$$ + \left( f_Y g_Y' + g_Y (f_X' \circ \opsqrt) + g_Y (g_X' \circ \opsqrt) \right) u_{\opsqrt}.$$
Now we examine the properties (i)-(iv) in Theorem \ref{genmaintheorem}.

Property (i) is false. 
Indeed, the ideal $I$ of $A \rtimes^{\sigma} G$
generated by $u_{\opabs}$ equals
$$A_X u_{\opabs} + A_X u_{\opsqr} + A_Y u_{\opsqrt}.$$
Hence $I$ is a nontrivial ideal of $A \rtimes^{\sigma} G$,
which, in particular, implies that 
$A \rtimes^{\sigma} G$ is not simple.

By direct inspection of the table
of partial composition for $P$
it follows that property (ii) is false.

Property (iii) is false.
In fact, if we let $S$ be any subset of 
the set of the non-negative real numbers,
then $\opabs(S) = S$.
Hence $P$ is not minimal.

Property (iv) is true. Indeed, the only nonidentity
function in $P$ that has equal domain and codomain
is $\opabs$. But $\opabs(x) = -x \neq x$ for
any negative real number.
\end{exmp}

\begin{rem}
It can be shown (see \cite[Example 28]{pajo11} for the details)
that the partially defined dynamical system in 
Example \ref{partialexmp} is not topologically free. 
\end{rem}

\begin{exmp}\label{examplediscrete}
Let $X$ denote a set equipped with the discrete topology.
We now consider two partially defined dynamical systems $P$ on $X$.

(a) Let $P$ be the partially defined dynamical system on $X$
having one-element subsets of $X$ as objects
and the unique functions between such sets as morphisms.
Now we examine the properties (ii), (iii) and (iv) in Theorem \ref{genmaintheorem}.
It is clear that $P$ is a locally abelian inverse connected
(small) groupoid so (ii) holds.
By the definition of $P$ it follows directly that 
it is both minimal and faithful.
Therefore, by  Theorem \ref{genmaintheorem},
we get that $A \rtimes^\sigma G$ is simple.
We leave it as an exercise to the reader
to show that  $A \rtimes^\sigma G$ is isomorphic as a complex algebra
to the direct limit $\varinjlim M_Y({\Bbb C})$,
where the direct limit is taken over finite subsets
$Y$ of $X$ and we let $M_Y({\Bbb C})$ denote 
the complex subalgebra of $A \rtimes^\sigma G$
generated by elements of the form $u_m$
for $m \in \mor(P)$ with $d(m),c(m) \in Y$.
The maps
$M_Y({\Bbb C}) \rightarrow M_{Y'}({\Bbb C})$,
for finite subsets $Y$ and $Y'$ of $X$
with $Y \subseteq Y'$, 
defining the direct limit,
are defined by sending $u_n$,
for $n \in P$ with $d(n),c(n) \in Y$,
to $u_n$.
If $X$ is a finite set of cardinality $n$,
then it is clear that
$\varinjlim M_Y({\Bbb C})$ is isomorphic
to the ring $M_n({\Bbb C})$ of $n \times n$ 
complex matrices.
In particular, we now retrieve simplicity
of $M_n({\Bbb C})$.

(b) Let $P$ be the partially defined dynamical 
system on $X$ consisting
of {\it all} subsets of $X$ as objects and 
{\it all} maps between such sets as morphisms.
Then $P$ is a small category which is not a groupoid.
Now we examine the properties (i)-(iv) in Theorem \ref{genmaintheorem}.
It is easy to see that $P$ is not topologically free.
By Proposition \ref{equivalentpajo11} and Proposition \ref{simplicitycommutativeA}
we conclude that $A \rtimes^\sigma G$ is not simple, so (i) is false.
By choosing $e,f\in \ob(G)$ (i.e. subsets of $X$)
of different cardinality, it is easy to see that $P$ is not inverse connected,
so (ii) is false.
It is clear that $P$ is minimal and faithful, so (iii) and (iv) are true.
\end{exmp}

\section*{Acknowledgements}
The second author was partially supported by The Swedish Research Council (postdoctoral fellowship no. 2010-918) and
The Danish National Research Foundation (DNRF) through the Centre for Symmetry and Deformation.


\begin{thebibliography}{99}

\bibitem{archbold93}
R. J. Archbold and J. S. Spielberg,
Topologically free actions and ideals
in discrete $C^*$-dynamical systems,
{\it Proc. of Edinburgh Math. Soc.} {\bf 37} (1993),
119--124.

\bibitem{bla06}
B. Blackadar,
{\it Operator Algebras: Theory of $C^*$-algebras and von
Neumann algebras}, Encyclopaedia of Mathematical Sciences, 122.
Operator Algebras and Non-commutative Geometry, III.
Springer-Verlag, Berlin, 2006.

\bibitem{davidson96}
K. R. Davidson, 
{\it $C^*$-algebras by Example}, Memoirs of the American
Mathematical Society, no. 75, American Mathematical Society,
Providence RI, 1996.

\bibitem{effros67}
E. G. Effros and F. Hahn, {\it Locally compact
transformation groups and $C^*$-algebras}, Memoirs of the American
Mathematical Society, no. 75, American Mathematical Society,
Providence RI, 1967.

\bibitem{elliot80}
G. A. Elliott,
Some simple $C^*$-algebras constructed as crossed products
with discrete outer automorphism groups,
{\it Publ. Res. Inst. Math. Sci.}
{\bf 16} (1980), 299--311.

\bibitem{exelvershik06}
R. Exel and A. Vershik,
C*-algebras of irreversible dynamical systems,
{\it Canad. J. Math.} {\bf 58}(1) (2006), 39--63.

\bibitem{kawamura90}
S. Kawamura and J. Tomiyama,
Properties of topological dynamical
systems and corresponding $C^*$-algebras,
{\it Tokyo. J. Math.} {\bf 13} (1990), 251--257.

\bibitem{kishimoto81}
A. Kishimoto,
Outer automorphisms and reduced crossed products of
simple $C^*$-algebras,
{\it Comm. Math. Phys.} {\bf 81} (1981), 429--435.

\bibitem{liu06}
G. Liu and F. Li,
On Strongly Graded Rings and the
Corresponding Clifford Theorem,
{\it Algebra Colloq.} {\bf 13} (2006),
181--196.

\bibitem{lu06}
P. Lundstr\"{o}m, Separable Groupoid Rings,
{\it Comm. Algebra} {\bf 34} (2006), 3029--3041.

\bibitem{pajo11}
P. Lundstr\"{o}m and J. \"{O}inert,
Skew category algebras associated with partially defined dynamical systems,
{\it Internat. J. Math.} {\bf 23}(4) (2012), pp. 16.

\bibitem{masuda1}
T. Masuda,
Groupoid dynamical systems and crossed product. I. The case of W*-systems.
{\it Publ. Res. Inst. Math. Sci.} {\bf 20}(5) (1984), 929--957.

\bibitem{masuda2}
T. Masuda,
Groupoid dynamical systems and crossed product. II. The case of C*-systems.
{\it Publ. Res. Inst. Math. Sci.} {\bf 20}(5) (1984), 959--970.

\bibitem{mun00}
J. R. Munkres,
{\it Topology}, Second Edition,
Prentice-Hall, 2000.

\bibitem{mur36} F. J. Murray and J. von Neumann, On rings of
operators, {\it Ann. of Math.} {\bf 37} (1936), 116--229.

\bibitem{mur43} F. J. Murray and J. von Neumann, On rings of
operators IV, {\it Ann. of Math.} {\bf 44} (1943), 716--808.

\bibitem{von61}
J. von Neumann, Collected works vol. III. Rings
of operators, Pergamon Press, 1961.

\bibitem{oinert09}
J. \"{O}inert,
Simple Group Graded Rings and Maximal Commutativity,
Contemporary Mathematics (American Mathematical Society),
Vol. 503, pp. 159-175 (2009).

\bibitem{oinlun08}
J. \"{O}inert and P. Lundstr\"{o}m,
Commutativity and Ideals in Category Crossed Products,
{\it Proc. Est. Acad. Sci.} {\bf 59}(4) (2010), 338--346.

\bibitem{oinlun10}
J. \"{O}inert and P. Lundstr\"{o}m,
The Ideal Intersection Property for Groupoid Graded Rings,
{\it Comm. Algebra} {\bf 40}(5) (2012), 1860--1871.

\bibitem{oinlunMiyashita}
J. \"{O}inert and P. Lundstr\"{o}m,
Miyashita Action in Strongly Groupoid Graded Rings,
{\it Int. Electron. J. Algebra}, Vol. 11 (2012), 46--63.

\bibitem{oinricsil11}
J. \"{O}inert, J. Richter and S. D. Silvestrov,
Maximal commutative subalgebras and simplicity of Ore extensions.
Preprint available at arXiv:1111.1292v1 [math.RA]

\bibitem{oin06}
J. \"{O}inert and S. D. Silvestrov,
Commutativity and Ideals in Algebraic Crossed Products,
{\it J. Gen. Lie T. Appl.} {\bf 2}, no. 4, 287--302 (2008).

\bibitem{oin07}
J. \"{O}inert and S. D. Silvestrov,
On a Correspondence Between Ideals and Commutativity in
Algebraic Crossed Products, {\it J. Gen. Lie T. Appl.} {\bf 2},
no. 3, 216--220 (2008).

\bibitem{oin08}
J. \"{O}inert and S. D. Silvestrov,
Crossed Product-Like and Pre-Crystalline Graded Rings,
Chapter 24 in 'Generalized Lie Theory in Mathematics, Physics and Beyond', Silvestrov, S.; 
Paal, E.; Abramov, V.; tolin, A. (Eds.), 16pp., Springer (2009).

\bibitem{oin10}
J. \"{O}inert, S. Silvestrov, T. Theohari-Apostolidi and H. Vavatsoulas,
Commutativity and Ideals in Strongly Graded Rings,
{\it Acta Appl. Math.} {\bf 108}, no. 3, 585-602 (2009).

\bibitem{oinertarxiv11}
J. \"{O}inert,
Simplicity of Skew Group Rings of Abelian Groups.
Preprint available at arXiv:1111.7214v2 [math.RA]

\bibitem{paterson99}
A. L. T. Paterson,
{\it Groupoids, inverse semigroups, and their operator algebras},
Progress in Mathematics, 170. Birkhäuser Boston, Inc., Boston, MA, 1999.

\bibitem{pedersen79}
G. K. Pedersen,
{\it $C^*$-algebras and their Automorphism
Groups}, Academic Press, 1979.

\bibitem{power78}
S. C. Power, Simplicity of $C^*$-algebras of
minimal dynamical systems, {\it J. London Math. Soc.}
{\bf 18} (1978), 534-538.

\bibitem{quigg92}
J. C. Quigg and J. S. Spielberg,
Regularity and hyporegularity in $C^*$-dynamical system,
{\it Houston J. Math.} {\bf 18} (1992), 139--152.

\bibitem{spielberg91}
J. S. Spielberg,
Free-product groups, Cuntz-Krieger algebras, and
covariant maps,
{\it Internat. J. Math.} {\bf 2} (1991), 457--476.

\bibitem{svesildejeu07}
C. Svensson, S. Silvestrov and M. de Jeu,
Dynamical Systems and Commutants in Crossed Products,
{\it Internat. J. Math.} {\bf 18}(4) (2007), 455--471.

\bibitem{svesildejeu08}
C. Svensson, S. Silvestrov and M. de Jeu,
Connections between dynamical systems and crossed
products of Banach algebras by $\mathbb{Z}$, in
{\it Methods of Spectral Analysis in Mathematical Physics},
Operator Theory: Advances and Applications, Vol. 186,
Birkhauser, 2009, pp. 391--402.

\bibitem{svesildejeu09}
C. Svensson, S. Silvestrov and M. de Jeu,
Dynamical Systems Associated with Crossed Products,
{\it Acta Appl. Math.} {\bf 108} (2009), 547--559.

\bibitem{takesaki03}
M. Takesaki,
{\it Theory of operator algebras. II},
Encyclopaedia of Mathematical Sciences, 125. Operator Algebras and
Non-commutative Geometry, 6. Springer-Verlag, Berlin, 2003.

\bibitem{tomiyama87} J. Tomiyama, Invitation to $C^*$-algebras and
topological dynamics, World Sci., Singapore, New Jersey, Hong Kong,
1987.

\bibitem{tomiyama92} J. Tomiyama, The interplay between topological
dynamics and theory of $C^*$-algebras. Lecture Notes Series, 2.
Global Anal. Research Center, Seoul, 1992.

\bibitem{williams07} 
D. P. Williams, 
Crossed Products of $C^*$-algebras,
Mathematical Surveys and Monographs,
American Mathematical Society,
Province, R. I., 2007.

\bibitem{zel68}
G. Zeller-Meier,
Produits crois\'{e}s d'une $C^*$-alg\`{e}bre par un groupe
d'automorphismes, {\it J. Math. pures et appl.} {\bf 47} (1968), 101--239.

\end{thebibliography}
\end{document}